\documentclass[a4paper,12pt, reqno]{amsart}
\usepackage{tikz}
\usepackage{graphicx}
\usepackage{array}
\usetikzlibrary{cd}
\usepackage{}
\usepackage{amssymb, amsmath, amsfonts, amscd}
\usepackage{mathrsfs, latexsym, array, longtable}
\usepackage{xcolor, bm, enumitem}
\usepackage{hyperref}

\newtheorem{theorem}{Theorem}[section]
\newtheorem{proposition}[theorem]{Proposition}
\newtheorem{lemma}[theorem]{Lemma}
\newtheorem{question}[theorem]{Question}
\newtheorem{remark}[theorem]{Remark}
\newtheorem{definition}[theorem]{Definition}
\newtheorem{eg}[theorem]{Example}

\title{\textbf{On pluricanonical boundedness of varieties of general type}}
\author{Pengjin Wang}
\date{\today}
\address{\rm Pengjin Wang, Shanghai Center for Mathematical Sciences, Fudan University, Jiangwan Campus, Shanghai, 200438, China}
\email{pjwang25@m.fudan.edu.cn}
\begin{document}

\begin{abstract}
    We present a new proof of a theorem of Chen and Jiang: for any integer $n>1$, there is a constant $K_n>0$ such that every smooth projective $n$-fold $X$ with $\operatorname{vol}(X)>K_n$ has either the stable birational $2$-canonical map or a M$^c$Kernan fibration. This amends a gap in the original proof. As a direct application of our method, we improve a former boundedness theorem of Lacini and prove that for any integer $r>1$ and $n\geq 1$, $r$-canonical maps of $n$-folds of general type have birationally bounded fibers. This gives an affirmative answer to a question posed by Chen and Jiang in 2014.
\end{abstract}
\maketitle
\tableofcontents

\pagestyle{myheadings}
\markboth{\hfill Pengjin Wang\hfill}{\hfill On pluricanonical boundedness of varieties of general type\hfill}
\numberwithin{equation}{section}

\section{Introduction}
Throughout we work over the complex number field $\mathbb{C}$.

Understanding pluricanonical systems of smooth projective varieties of general type is an important task in birational geometry. By the work of  Hacon, M$^c$Kernan \cite{hacon2006boundedness} and Takayama \cite{takayama2006pluricanonical}, for any positive integer $n$, there exists an optimal positive integer $r_n$ such that the $r$-canonical map $\varphi_{r,X}:=\phi_{\lvert rK_X \rvert}$ is birational onto its image for any smooth projective $n$-fold $X$ of general type and any positive integer $r\geq  r_n$. We hope to calculate the value of $r_n$, while it is only known that $r_1=3$, $r_2=5$ by Bombieri \cite{bombieri1973canonical} and $27\leq r_3 \leq 57$ by Iano-Fletcher\cite{iano2000working} and Chen-Chen\cite{chen2015explicit,chen2016minimal}.

Recently, Chen and Liu proved the following theorem: 
\begin{theorem}(\cite[Theorem 1.1]{chen2024lifting})\label{caonical index}
    Let $n>1$ be an integer. There exists a positive number $K$ such that, if $X$ is a smooth projective $n$-fold with $\operatorname{vol}(X)>K$, then the $r$-canonical map $\varphi_{r,X}$ is birational for all $r\geq r_{n-1}$.
\end{theorem}

When $n=3$, it is just \cite[Theorem 1.2]{todorov2007pluricanonical}. When $n=4$, it is just \cite[Theorem 1.4]{chen2017reduction}. The proof of Theorem \ref{caonical index} essentially consists of two parts: one is \cite[Theorem 6.8]{chen2017reduction}, and the other is \cite[Theorem 1.5]{chen2024lifting}. However, the original proof of \cite[Theorem 6.8]{chen2017reduction} contains a gap. In this article, we will give a new proof of \cite[Theorem 6.8]{chen2017reduction} using the tool introduced in \cite{lacini2023boundedness}.

The main purpose of this paper is to prove the following statement, which is exactly \cite[Theorem 6.8]{chen2017reduction}.
\begin{theorem}\label{main theorem}
   Let $n>1$ be an integer. Fix a function $\lambda: \mathbb{Z}_{>0} \times \mathbb{Z}_{>0} \rightarrow \mathbb{R}_{>0}$. There exist integers $M_{n-1}>M_{n-2}>\ldots>M_1>0$ and a constant $K>0$ such that, for any smooth projective $n$-fold $X$ with $\operatorname{vol}(X) \geq K$, the pluricanonical map $\varphi_{a,X}$ of $X$ is birational for $a \geq 2$, unless that, after birational modifications, $X$ admits a fibration $f: X \rightarrow Z$ which satisfies $(B)_{\mathfrak{X}_{k, M_k^k}, \lambda\left(k, M_k^k\right)}$ for some $1 \leq k \leq n-1$.
\end{theorem}

We now explain the notation used in the above statement. For $M>0$ and $n\in \mathbb{Z}_{>0}$, define $\mathfrak{X}_{n,M}$ to be the set of smooth projective $n$-folds of general type with $0<\operatorname{vol}(X)\leq M$, which is birationally bounded by Hacon-M$^c$Kernan (see Theorem \ref{HaconMckernanbdd}).
In addition, we have the following definition:

\begin{definition}(\cite[Definition 6.5]{chen2017reduction})
Given a birationally bounded set $\mathfrak{X}$ of smooth projective varieties and given a positive number $c$, we say that a fibration $f:X \to T$ between smooth projective varieties satisfies condition $(\mathrm{B})_{\mathfrak{X}, c}$ if
\begin{enumerate}
    \item a general fiber $F$ of $f$ is birationally equivalent to an element of $\mathfrak{X}$;
    \item for a general point $t \in T$, there exists an effective $\mathbb{Q}$-divisor $D_t$ with $D_t \sim_{\mathbb{Q}} \varepsilon K_X$ for a positive rational number $\varepsilon < c$, such that the fiber $F_t = f^{-1}(t)$ is an irreducible component of $\operatorname{Nklt}(X, D_t)$.
\end{enumerate}
\end{definition}

We now explain the gap in the original proof of \cite[Theorem 6.8]{chen2017reduction}. Indeed, we cannot ensure that $t_{d+1}<2t_d+\epsilon$ in line 7 on page 2079 except when $d=1$. When $d=1$, the following key lemma makes $t_{d+1}<2t_d+\epsilon$, which plays an important role in the proof of \cite[Theorem 1.2]{todorov2007pluricanonical} and \cite[Theorem 1.3]{chen2017reduction}.

\begin{lemma}(\cite[Lemma 3.3]{todorov2007pluricanonical})
Let $Y$ be a smooth projective variety and let $(\Delta_i, W_i)$ for $i = 1,2$ be a pair such that at some point $y \in Y$ we see that $W_i$ is an exceptional lc center of codimension $1$ at $y$ for $K_Y + \Delta_i$, with $\Delta_i$ smooth at $y$. Then there exists $Z \subset W_1 \cap W_2$, a minimal pure lc center at $y$ for a pair $(Y, \Delta)$ with $\Delta = k(\Delta_1 + \Delta_2)$ for some rational number $0 < k \leq 1$.
\end{lemma}

When the lc center is not of codimension one, $\Delta$ cannot be smooth at $y$. After removing this assumption on smoothness, the lemma above still fails so this approach to cut down lc centers cannot work when $d> 1$. Consider the following example:
\begin{eg}
    Let $Y=\mathbb{P}^4$ and $y=[0:0:0:0:1]$ with coordinates $z_1,\ldots,z_5$. Let $S_1$ (resp., $S_2$) be the plane defined by $z_1=z_2=0$ (resp., $z_1=z_3=0$). Let $H_1$, $H_2$, $H_3$ (resp., $H_4$, $H_5$, $H_6$) be general hyperplanes passing through $S_1$ (resp., $S_2$). Let $H$ be a general hyperplane passing through $y$ with the intersection $l_1$ (resp., $l_2$) with $S_1$ (resp., $l_2$). Let $l_3$ be the intersection of $S_1$ with $S_2$.

   Let $\Delta_1=\frac{2}{3}H_1+\frac{2}{3}H_2+\frac{2}{3}H_3+\frac{5}{6}H$ and $\Delta_2=\frac{2}{3}H_4+\frac{2}{3}H_5+\frac{2}{3}H_6+\frac{5}{6}H$. Then $(\mathbb{P}^4,\Delta_i)$ is lc at $y$ and has the unique lc center $S_i$ for $i=1,2$. Consider the pair $(\mathbb{P}^4,c(\Delta_1+\Delta_2))$ for $c>0$. To make $l_3$ (resp., $\{y\}$, $l_i$ for $i=1,2$) a non-klt center, there should be $c\geq \frac{3}{4}$(resp., $\frac{12}{17}$, $\frac{9}{11}$).
    Therefore, $(\mathbb{P}^4,\frac{3}{5}(\Delta_1+\Delta_2))$ is lc at $y$ and has the unique lc center $H$. When $c<\frac{3}{5}$ (resp., $c>\frac{3}{5}$), the pair will be klt (resp., not lc) at $y$. 
\end{eg}
 One may also refer to the proof of \cite[Lemma 4.2]{mckernan2002boundedness}. Therefore, we cannot cut down lc centers in this way in high dimensions. Instead, we will present a new proof using the inductive method in \cite{lacini2023boundedness}.

As an application of Theorem \ref{main theorem}, we consider the boundedness of fibers for pluricanonical maps of varieties of general type.

For smooth projective surfaces of general type, Beauville showed in \cite{beauville1979application}, that there is a universal bound on the degree of 1-canonical maps (if generically finite) or volume of the general fiber of 1-canonical maps (if not generically finite) if $p_g>0$. For higher dimensions, the result fails to hold due to examples in \cite{derek2004degree} and \cite{chen2014unboundedness}. The example in \cite{derek2004degree} implies that the $1$-canonical map can be generically finite of arbitrarily large degree. The example in \cite{chen2014unboundedness} implies that the general fiber of the $1$-canonical map can be of any dimension $0<d<n$ and have arbitrarily large volume.

Naturally, Chen and Jiang asked whether the boundedness result holds for $r$-canonical maps when $r>1$.
\begin{question}(\cite[Question 2.5]{chen2014unboundedness})\label{question}
    Let $r>1$ and $n\geq 3$ be two integers. Considering all those projective minimal $n$-folds of general type which are $r$-canonically fibered, are the birational invariants of fibers of $r$-canonical maps universally bounded from above?
\end{question}

For $n=3$, Lacini has given an affirmative answer in \cite{lacini2023boundedness}. For $n\geq 4$, Lacini has proven the result for $r\geq (n-2)r_{n-2}+2$ by \cite[Theorem 1.2]{lacini2023boundedness}. In this article, we will use Theorem \ref{main theorem} to improve Lacini's boundedness theorem, so that we can prove the result for any $r>1$.

 Let us first recall the definition in \cite[Definition 4.2]{lacini2023boundedness}.
Let $\mathcal{X}$ be a set of pairs $(X, \varphi_X)$, where $X$ is a smooth projective variety, and $\varphi_X$ is a rational map whose source is $X$. We say that maps $\varphi_X$ \textit{have birationally bounded fibers}, if for every $X$ the general fiber of $\varphi_X$ belongs to a fixed birationally bounded family. This is a reformulation of the boundedness referred to above.

We give an affirmative answer to Question \ref{question} by proving the following theorem:

\begin{theorem}\label{main theorem2}
    Let $n\geq1$ and $r\geq 2$ be positive integers. If the $r$-th plurigenus is not zero, $r$-canonical maps of smooth projective $n$-folds of general type have birationally bounded fibers.
\end{theorem}

\section{Preliminaries}

\subsection{Pairs and singularities}
A \textit{sub-pair} $(X, \Delta)$ consists of a normal quasi-projective variety $X$ and an $\mathbb{R}$-divisor $\Delta$ such that $K_X + \Delta$ is $\mathbb{R}$-Cartier. A sub-pair $(X, \Delta)$ with $\Delta \geq 0$ is called a \textit{pair}. If a pair $(X, \Delta)$ satisfies that $X$ is smooth and that the support of $\Delta$ is of simple normal crossing, we say that $(X, \Delta)$ is a \textit{log smooth pair}.

Let $(X, \Delta)$ be a sub-pair. Let $f: Y \to X$ be a log resolution of $(X, \Delta)$ and write $f^\ast(K_X + \Delta)=K_Y + \Gamma $. Then define the \textit{log discrepancy} of a prime divisor $D$ on $Y$ with respect to the sub-pair $(X, \Delta)$ to be
$$a(D, X, \Delta) := 1 - \mu_D(\Gamma).$$
We say $(X, \Delta)$ is \textit{sub-lc} (resp., \textit{sub-klt}) if $a(D, X, \Delta) \geq 0$ (resp., $> 0$) for every prime divisor $D$ over $X$. If, in addition, $(X, \Delta)$ is a pair, we say $(X, \Delta)$ is \textit{lc} (resp., \textit{klt}). We say $(X, \Delta)$ is \textit{terminal} if $a(D, X, \Delta) > 1$ for any prime divisor $D$ which is exceptional over $X$.

Let $(X, \Delta)$ be a sub-pair. A \textit{non-klt place} is a prime divisor $D$ over $X$ with $a(D, X, \Delta) \leq 0$ and a \textit{non-klt center} is the image of a non-klt place on $X$. The \textit{non-klt locus} of $(X, \Delta)$ is the union of all non-klt centers of $(X, \Delta)$ which is denoted as $\mathrm{Nklt}(X, \Delta)$. An \textit{lc place} is a prime divisor $D$ over $X$ with $a(D, X, \Delta) = 0$ and an \textit{lc center} is the image of an lc place on $X$. We say that an lc center $G$ of $(X, \Delta)$ is a \textit{pure lc center} if $(X, \Delta)$ is lc at the generic point of $G$. We say that an lc center $G$ is an \textit{exceptional lc center} if there is a unique lc place over $X$ whose image is $G$.

Given a line bundle $L$ on a projective variety $X$, the graded ring of sections of $L$ is defined as the graded $\mathbb{C}$-algebra
$$R(X,L)=\bigoplus_{m\geq0}H^0(X,L^{\otimes m}).$$

\subsection{The volume of a line bundle}
\begin{definition}
   Let $X$ be a reduced equidimensional projective scheme of dimension $n$ and let $D$ be a Cartier divisor. The volume of $D$ is defined as:
    $$\operatorname{vol}(D)=\limsup_{m\to \infty}{\frac{n!h^0(X,\mathcal{O}_X(mD))}{m^n}}.$$
The canonical volume $\operatorname{vol}(X)$ of $X$ is defined as $\operatorname{vol}(K_{X^\prime})$, where $X^\prime$ is any smooth projective model of $X$.
\end{definition}
\begin{remark}
If $X$ is irreducible, it is just the classical definition. In this case, the lim sup computing $\operatorname{vol}(D)$ is actually a limit by \cite[Example 11.4.7]{lazarsfeld2}.

Since smooth projective varieties have terminal singularities, the canonical volume is well-defined and is a birational invariant.
\end{remark}
\begin{lemma}(\cite[Lemma 2.4]{lacini2023boundedness}=\cite[Section 5.3]{chen2017reduction})\label{fiber}
Let $X$ be a smooth projective variety of general type and dimension $n>1$. Fix a positive integer $r$ and let $V \subseteq H^0(X, rK_X)$ be a linear subspace with $\dim V\geq 2$. Let $\lvert \Lambda_V \rvert$ be the associated linear system. Assume also that $\operatorname{Mov}\lvert \Lambda_V\rvert$ is base point free and let $\varphi_V : X \to \mathbb{P}^{\dim(V)-1}$ be the associated morphism. Finally, let $d = \dim \varphi_V(X)$ and let $F$ be the general fiber of $\varphi_V$ (we don’t assume $F$ to be connected). Then
$$\operatorname{vol}(X) > \frac{1}{(rn)^n} \cdot \operatorname{vol}(F) \cdot (\dim(V)-d).$$
\end{lemma}
\begin{remark}
  When $\dim F=0$, we just replace $\operatorname{vol}(F)$  by $\operatorname{deg}\varphi_V$ and the proof still works.
\end{remark}

\subsection{Multiplier ideal sheaves}
\begin{definition}
    Given a pair $(X,\Delta)$, the multiplier ideal sheaf $\mathcal{I}(X,\Delta)$ is defined as follows. Take a log resolution $f:Y\rightarrow X$ of $(X,\Delta)$ and write $f^\ast(K_X+\Delta)=K_Y+\Gamma$, then
    $$\mathcal{I}(X,\Delta)=f_\ast(\mathcal{O}_Y(\lceil-\Gamma\rceil)).$$
\end{definition}

The multiplier ideal sheaf $\mathcal{I}(X,\Delta)$ is a nonzero coherent ideal sheaf and it does not depend on the choice of log resolutions. Furthermore, we have $\operatorname{Nklt}(X,\Delta)=\operatorname{Supp}\mathcal{O}_X/\mathcal{I}(X,\Delta)$. One can refer to \cite[Proposition 2.11.2]{kawamata2024algebraic} for details.

\begin{theorem}(\cite[Theorem 2.11.6]{kawamata2024algebraic})(Nadel vanishing theorem)
    Let $(X,\Delta)$ be a pair with $X$ projective and $L$ be a Cartier divisor on $X$. If $L-(K_X+\Delta)$ is nef and big, then $$H^p(X,\mathcal{O}_X(L)\otimes \mathcal{I}(X,\Delta))=0$$ for any $p\geq 1$.
\end{theorem}

The following criterion is a classical application of Nadel vanishing theorem, and we will use it to deduce birationality in next section.
\begin{lemma}(\cite[Lemma 2.17]{lacini2023boundedness})\label{nadel}
    Let $X$ be a smooth projective variety and let $D$ be a big integral divisor on $X$. Take a decomposition $D \sim_{\mathbb{Q}} A + B$ with $A$ an ample $\mathbb{Q}$-divisor and $B$ effective. Let $\lambda > 0$ be a rational number and let $x$ and $y$ be two points of $X$. Suppose that after possibly switching $x$ and $y$, there exists an effective $\mathbb{Q}$-divisor $\Delta \sim_{\mathbb{Q}} \lambda D$ such that $(X,\Delta)$ is lc at $x$, not klt at $y$ and $\{x\}$ is an isolated lc center. If $m>\lambda $ is a positive integer and $x,y\notin \operatorname{Supp}B\cup \operatorname{Bs}(K_X+mD)$, then $\lvert K_X + mD\rvert$ separates $x,y$.
\end{lemma}

\subsection{A birational family of tigers}
\begin{definition}(\cite[Definition 3.1]{mckernan2002boundedness})\label{mck tiger}
Let $X$ be a smooth projective variety and $D$ a $\mathbb{Q}$-divisor. We say that pairs of the form $\left(D_t, V_t\right)$ form a birational family of tigers of dimension $k$ and weight $w$ relative to $D$ if

(1) There is a projective morphism $f: Y \rightarrow B$ of normal projective varieties and an open subset $U$ of $B$ such that the fiber of $f$ over $t \in U$ is $V_t$.

(2) There is a morphism of $B$ to the Hilbert scheme of $X$ such that $B$ is the normalization of its image and $f$ is obtained by taking the normalization of the universal family.

(3) If $\pi: Y \rightarrow X$ is the natural morphism then $\pi\left(V_t\right)$ is a pure lc center of $(X,D_t)$.

(4) $\pi$ is birational.

(5) $0\leq D_t \sim_{\mathbb{Q}} \frac{1}{w} D$.

(6) The dimension of $V_t$ is $k$.
\end{definition}
When $D$ is big, by using the tiebreaking lemma (see \cite[Lemma 2.3]{chen2017reduction}), we can assume that $\pi\left(V_t\right)$ is an irreducible component of $\operatorname{Nklt}(X,D_t)$ after a small perturbation of $w$.

\subsection{An extension theorem}
\begin{definition}
    We say that a set $\mathfrak{X}$ of varieties is birationally bounded if there is a projective morphism $\tau:Z\rightarrow T$ between schemes, where $T$ is of finite type, such that every $X\in \mathfrak{X}$ is birational to $Z_t=\tau^{-1}(t)$ for some closed point $t\in T$.
\end{definition}
\begin{theorem}(\cite[Corollary 1.2]{hacon2006boundedness})\label{HaconMckernanbdd}
    For $n\in \mathbb{Z}_{>0}$ and $M\in \mathbb{R}_{>0}$, the set of smooth projective $n$-folds of general type with the canonical volume at most $M$ is birationally bounded.
\end{theorem}
We have the following type of extension theorem:
\begin{theorem}\label{extension}(\cite[Theorem 1.5]{chen2024lifting})
Let $n, d$ be two integers with $n > d > 0$ and $\mathcal{P}$ a birationally bounded set of smooth projective varieties of dimension $d$. Then there exists a positive number $t$, depending only on $d$ and $\mathcal{P}$, such that the following property holds:

Let $f : V \to T$ be a surjective fibration where $V$ and $T$ are smooth projective varieties, $\dim V = n$, $\dim T = n - d$ and $\operatorname{vol}(V) > 0$. Assume that the following conditions are satisfied:
\begin{enumerate}
    \item the general fiber $X$ of $f$ is birationally equivalent to an element of $\mathcal{P}$;
    \item there exists a positive rational number $\delta < t$ and an effective $\mathbb{Q}$-divisor $\Delta \sim_{\mathbb{Q}} \delta K_V$ such that $X$ is an irreducible component of $\operatorname{Nklt}(V, \Delta)$.
\end{enumerate}
Then the restriction map
$$H^0(V, pK_V) \to H^0(X_1, pK_{X_1}) \oplus H^0(X_2, pK_{X_2})$$
is surjective for any integer $p \geq 2$ and for any two different general fibers $X_1, X_2$ of $f$.
\end{theorem}

\section{A birational family of tigers with the separation property}
In this section, we recall the tool introduced by Lacini in \cite[Section 3]{lacini2023boundedness}.
\subsection{Separation Property}
\begin{definition}
    Let $X$ be a normal projective variety of dimension $n$ and let $D$ be a $\mathbb{Q}$-Cartier divisor. Let $f: Y \rightarrow B$ be a birational family of tigers of weight $w$ relative to $D$ and let $\pi: Y \rightarrow X$ be the associated birational morphism. We say that this family has the separation property if there is a non-empty open subset $U \subseteq Y$ such that for all $y_1, y_2 \in U$ there exists a $\mathbb{Q}$-divisor $D_{y_1, y_2}$ in $X$, which, after possibly switching $y_1$ and $y_2$, has the following properties:
    \begin{enumerate}
        \item $D_{y_1, y_2} \sim_{\mathbb{Q}} \lambda D$ with $\lambda<\frac{3}{w}$.
        \item $(X, D_{y_1, y_2})$is lc but not klt at $\pi\left(y_1\right)$.
        \item $(X, D_{y_1, y_2})$ has the unique lc center $V$ at $\pi\left(y_1\right)$ and $V \subseteq \pi\left(f^{-1}\left(f\left(y_1\right)\right)\right)$.
        \item $\left(X, D_{y_1, y_2}\right)$ is not klt at $\pi\left(y_2\right)$.
    \end{enumerate}
\end{definition}
\begin{remark}
    Note that $V$ can be strictly contained in $\pi\left(f^{-1}\left(f\left(y_1\right)\right)\right)$.
\end{remark}
\begin{proposition}\label{birational}
Let $X$ be a smooth projective variety and $D$ an integral big divisor. If there is a birational family of tigers of weight $w$ and dimension zero relative to $D$ that has the separation property, then the linear system $\lvert K_X+mD \rvert$ induces a birational map for any integer $m> \frac{3}{w}$. 
\end{proposition}
\begin{proof}
    This follows directly from Lemma \ref{nadel}.
\end{proof}

\subsection{Intrinsic Tigers} 
In order to obtain a birational family of tigers with the separation property, we construct a concrete birational family of tigers using base loci. It will satisfy the separation property.
\begin{definition}\label{base loci}
Let $X$ be a smooth projective variety of dimension $n$ and $D$ a $\mathbb{Q}$-divisor on $X$. For a positive integer $k$ and a point $x \in X$ we define
$$S_x^k(D)=\left\{H \in \lvert \lfloor D\rfloor \rvert \text { such that }\operatorname{mult}_x H \geq k\right\}$$
and define $B_x^k(D)$ to be the union of the irreducible components of $\bigcap_{H \in S_x^k(D)} H$ passing through $x$. 

Choose any affine open set $U \subseteq X$ and consider the diagonal $\Delta \subseteq X \times U$. Then we define
$$\mathcal{S}^k(D)=\left\{H \in\left|\operatorname{pr}_1^*(\lfloor D \rfloor)\right| \text { such that } \operatorname{mult}_{\Delta} H \geq k\right\}
$$ and define $\mathcal{B}^k(D)$ to be the union of the irreducible components of $\bigcap_{H \in \mathcal{S}^k(D)} H$ containing $\Delta$. 
\end{definition}
\begin{remark}\label{remark}
By the following lemma, after possibly shrinking $U$, we have that
$$
B_x^k(D)=\operatorname{pr}_1\left(\mathcal{B}^k(D) \cap (X \times\{x\})\right)
$$
for all $x \in U$. By the above description, we may assume that $\operatorname{dim} B_x^k(D)$ is a constant for all $x \in U$, and we denote its value by $d(D, k)$.
\end{remark}
\begin{lemma}\label{mult} (\cite[Corollary 5.2.12]{lazarsfeld1})
    Let $p: X \rightarrow T$ be a morphism of smooth varieties and let $D$ be a $\mathbb{Q}$-divisor on $X$. Then for a general point $t \in T$,
$$
\operatorname{mult}_y(X, D)=\operatorname{mult}_y\left(X_t, D_t\right)
$$
for every $y \in X_t$, where $X_t$ is the fiber of $X$ over $t$ and $D_t$ is the restriction of $D$ to $X_t$.
\end{lemma}

\begin{definition}\label{intrinsic}
    Let $X$ be a smooth projective variety of dimension $n$, $D$ a $\mathbb{Q}$-divisor on $X$ and $w>0$ a positive rational number. We say that a morphism of smooth projective varieties $f: Y \rightarrow B$ is a birational family of intrinsic tigers of weight $w$ relative to $D$ if there are positive integers $k$ and $l>0$, a nonempty open subset $U\subseteq Y$ and a birational morphism $\pi: Y \rightarrow X$ such that
    
(1) $d(lD, k)<n$.

(2) $\pi\left(f^{-1}(f(y)\right)=B_{\pi(y)}^k(l D)$ for $y\in U$.

(3) Every element of $S_{\pi(y)}^k(lD)$ has multiplicity at least $wl(n-d(lD, k))$ along $B_{\pi(y)}^k(lD)$.
\end{definition}
Note that by connectedness of $B_x^k$ and generic smoothness, $B_{\pi(y)}^k(l D)$ is irreducible for $y\in U$.

\begin{lemma}
   Let $f$ be a birational family of intrinsic tigers of weight $w$ relative to $D$, with $l$ and $k$ as in Definition \ref{intrinsic}.

   When $wl>1$, $f$ is a birational family of tigers of weight $w$ relative to $D$ as in definition \ref{mck tiger}. When $wl>3$, $f$ has the separation property.
\end{lemma}
\begin{proof}
   This follows directly from \cite[Lemma 3.3, Lemma 3.4]{lacini2023boundedness}.
\end{proof}
\begin{remark}
    Indeed, $w$ will be sufficiently large later in practice so that the condition $wl>3$ will be automatically satisfied.
\end{remark}

For the classical case, when $D$ has sufficiently large volume, there is a birational family of tigers of large weight as in \cite[Lemma 3.3]{mckernan2002boundedness}. This property also holds for intrinsic tigers: 

\begin{lemma}(\cite[Lemma 3.10]{lacini2023boundedness})\label{construct intrinsic}
    Let $X$ be a smooth projective variety of dimension $n$ and let $D$ be a big $\mathbb{Q}$-divisor. If $\mathrm{vol}(D)>\left(2^n a\right)^n$ for some positive rational number $a>0$, then there is a birational family of intrinsic tigers of weight $\frac{a}{n}$ relative to $D$.
\end{lemma}

\subsection{A good refinement to cut down lc centers}
\begin{definition}
  Let $X$ be a smooth projective variety of dimension $n$ and let $D$ be a $\mathbb{Q}$-divisor. Let $f: Y \rightarrow B$ be a birational family of tigers of weight $w$ relative to $D$ that has the separation property. We say that $f$ admits a good refinement $f^{\prime}$ of weight $u$ and dimension $d$ if there is a $\pi$-exceptional divisor $E$, and a commutative diagram
\begin{center}
        \begin{tikzcd}
            Y^\prime \arrow{r}{\tau} \arrow{d}{f^\prime} &Y \arrow{d}{f}\\
            C \arrow{r}{g} & B
        \end{tikzcd}
\end{center}
and a nonempty open subset $U^\prime\subseteq Y^\prime$ such that:

(1) $\tau\vert_{U^\prime}$ is an isomorphism.

(2) For general $b \in B, f_b^{\prime}=f^\prime \vert_{Y_{b^\prime}}: Y_b^{\prime} \rightarrow C_b$ is a birational family of tigers of weight $u$ relative to $\left(\pi^\ast D+E\right)\vert_{Y_b}$ that has the separation property on $U^\prime\cap Y_b^\prime$.

(3) The general fiber of $f^{\prime}$ has dimension $d$.
\end{definition}
The following lemma plays a key role in dimension induction.
\begin{lemma}(\cite[Lemma 3.6]{lacini2023boundedness})\label{refinement}
    Let $X$ be a smooth projective variety of dimension $n$ and let $D$ be a big $\mathbb{Q}$-divisor whose graded ring of sections is finitely generated. Let $f: Y \rightarrow B$ be a birational family of tigers of weight $w$ relative to $D$ having the separation property. Suppose furthermore that $f$ admits a good refinement $f^{\prime}$ of weight $u$. Then $f^{\prime}$ is a birational family of tigers of weight $\frac{1}{\frac{1}{u}+\frac{1}{w}}$ relative to $D$ that has the separation property.
\end{lemma}
\begin{remark}
     By \cite[Example 2.1.31]{lazarsfeld1}, the condition that the graded ring of sections is finitely generated implies that there is a birational morphism $\mu:X^\prime\rightarrow X$, an effective divisor $N$ on $X^\prime$ and a positive integer $p$ such that $D^\prime=\mu^\ast(pD)-N$ is base point free and $R(X^\prime,D^\prime)=R(X,D)^{(p)}$. Take the Iitaka fibration $\phi:X^\prime\rightarrow X^{\prime\prime}$ relative to $D$ and let $A$ be an ample divisor on $X^\prime $ such that $D^\prime=\phi^\ast A$. Since $R(X^{\prime\prime},A)=R(X,D)^{(p)}$, we can cut down lc centers on an open subset of $X$ assuming that $D$ is ample. See \cite[Lemma 2.16]{lacini2023boundedness} for details.

    We can use Lemma \ref{refinement} for $D=K_X$ on a smooth projective variety of general type thanks to \cite{birkar2010existence}.
\end{remark}

\section{Proof of Theorem \ref{main theorem}}
\begin{lemma}\label{good algebraic family}
    Let $X$ be a smooth projective $n$-fold of general type and $f:Y\rightarrow B$ a birational family of tigers of dimension $d\geq 1$ relative to $K_X$ with the separation property. Let $\pi:Y\rightarrow X$ be the associated birational morphism.

    Let $\alpha$ be a positive rational number. If the volume of the general fiber is larger than $\alpha^d$, then $f$ admits a good refinement $f^{\prime}$ of weight $\frac{\alpha}{2^d\cdot d}$.
\end{lemma}
\begin{proof}
    It is essentially a part of the proof of \cite[Lemma 3.11]{lacini2023boundedness}. We rephrase just for the convenience of readers.

    Firstly, we show that we can obtain a birational family of intrinsic tigers relative to $K_{Y_b}$ with the common weight $w$ and common positive integers $k,l$ on $Y_b$ for general $b\in B$. By generic flatness, there is an open subset $V$ of $B$ such that $f$ is flat over $V$. By the semicontinuity theorem \cite[Chapter III, Theorem 12.8]{hartshorne1977algebraic}, given a positive integer $r$ there is a closed subset $W_r$ of $V$ consisting of $b\in V$ such that $h^0(Y_b,rK_{Y_b})\geq \frac{(\alpha r)^d}{d!}$.
    Since the volume of the general fiber is larger than $\alpha$, the union of all $W_r$, for $r\in\mathbb{Z}_{>0}\cap \frac{2^{d+1}}{\alpha} \mathbb{Z}$, is equal to $V$, which implies there is an $l\in \mathbb{Z}_{>0}\cap \frac{2^{d+1}}{\alpha} \mathbb{Z}$ such that $W_l=V$.
    We can take $l$ to be sufficiently large so that $$h^0(Y_b,lK_{Y_b})\geq \frac{(\alpha l)^d}{d!}>\binom{d+\frac{\alpha l}{2}-1}{d}.$$
    After possibly shrinking $V$, we can assume that $h^0(Y_b,lK_{Y_b})$ is a constant for $b\in V$. By Grauert's theorem (\cite[Corollary 12.9]{hartshorne1977algebraic}), letting $U=f^{-1}(V)$ we have the surjection
    $$H^0(U,lK_U)\longrightarrow H^0(Y_b,lK_{Y_b})$$
    for $b\in V$. After possibly shrinking $V$ we can assume that for $1\leq i\leq 2^d\cdot \frac{\alpha l}{2^{d+1}}$, $d(lK_{Y_b},i)$ is a constant on $V$. The proof of Lemma \ref{construct intrinsic} implies that we can choose common $w=\frac{\alpha}{2^d\cdot d}$, $l$ and $k$ for $b\in V$.

    Now we show that the base loci $B_x^k\left(l K_{Y_b}\right)$ form a nice algebraic family as in the definition of good refinement. Consider $\mathcal{F}=\{(x, y) \in U \times U$ such that $f(x)=f(y)\}$. Then consider the sections of $\operatorname{pr}_1^*(lK_U)$ on $U \times U$ that are tangent to order at least $k$ to $\mathcal{F}$ along $\Delta \subseteq U \times U$. Let $\mathcal{C}^k(l K_U)$ be their base locus. By Remark \ref{remark} and Lemma \ref{mult}, we have that $B_x^k\left(l K_{Y_b}\right)=\operatorname{pr}_1\left(\mathcal{C}^k(lK_U) \cap (U \times\{x\})\right)$ for general $x \in U$. This gives a good refinement of weight $\frac{\alpha}{2^d\cdot d}$.
\end{proof}

\begin{proof}[Proof of Theorem \ref{main theorem}]
    For each $1\leq d \leq n-1$, we define $M_d$ in the following way. First, take $M_1$ to be a positive integer strictly larger than $3\cdot 2^{n+1}\cdot n^2$. Inductively, if $M_{d-1}$ is defined, take a positive integer 
    $$M_d>\max\{ M_{d-1},\frac{2^{n+2}\cdot n^2}{\lambda(d-1,M_{d-1}^{d-1})}\}.$$
    Besides, let $K$ satisfy $$\frac{2^n\cdot n}{K^\frac{1}{n}}<\frac{1}{6}\max_{1\leq k\leq n-1} \{1,\lambda(k,M_k^k) \}.$$

    By Lemma \ref{construct intrinsic}, we get a birational family of tigers with the separation property of dimension $n_1\leq n-1$ and weight $\frac{1}{t_1}$ relative to $K_X$, where $0<t_1< \frac{2^n\cdot n}{\operatorname{vol}(X)^\frac{1}{n}}$. 

    The induction goes as follows. Assume that there is a birational family of tigers $f:Y\rightarrow B$ with the separation property of dimension $n_d\leq n-d$ and weight $\frac{1}{t_d}$ relative to $K_X$. Then for a general fiber $Y_b$,

    $(1)_d$ either $\operatorname{vol}(Y_b)\geq M_{n_d}^{n_d}>(\frac{ M_{n_d}}{2})^{n_d}$, then by Lemma \ref{good algebraic family} and Lemma \ref{refinement}, we get a birational family of tigers with the separation property of dimension $n_{d+1}\leq n-d-1$ and weight $\frac{1}{t_{d+1}}$ relative to $K_X$, where 
    $$t_{d+1}\leq t_d+\frac{2^{n_d+1}\cdot n_d}{M_{n_d}}.$$

    $(2)_d$ or $\operatorname{vol}(Y_b)< M_{n_d}^{n_d}$ and $f:Y\rightarrow B$ satisfies the condition $(B)_{\mathfrak{X},t_d}$ for $\mathfrak{X}=\mathfrak{X}_{n_d,M_{n_d}^{n_d}}$.

    The induction stops when arriving at $(2)_d$ for some $d$, or we run the induction to the end and get a biational family of tigers with the separation property of dimension zero and weight $\frac{1}{t_m}$ relative to $K_X$.

    In the first case, we have 
    \begin{equation*}
    \begin{aligned}
        t_{d} &\leq t_1+\sum_{i=1}^{d-1} \frac{2^{n_{i}+1}\cdot n_{i}}{M_{n_{i}}}\\
        &< \frac{2^n\cdot n}{\operatorname{vol}(X)^\frac{1}{n}}+(d-1)\cdot \frac{2^{n+1}\cdot n}{M_{n_{d}+1}}\\
        &\leq \frac{2^n\cdot n}{K^\frac{1}{n}}+\frac{2^n\cdot n^2}{M_{n_{d}+1}}\\
        &\leq \lambda(n_d,M_{n_d}^{n_d}),
    \end{aligned}
    \end{equation*}
    where the last inequality comes from the choice of $K$ and $M_i$ at the beginning of the proof.
    So after birational modifications, $X$ admits a fibration which satisfies $(B)_{\mathfrak{X},c}$ for $\mathfrak{X}=\mathfrak{X}_{k,M_{k}^{k}}$ and $c=\lambda(k,M_{k}^{k})$ for some $1\leq k\leq n-1$.

    In the second case, we have 
    \begin{equation*}
        \begin{aligned}
           t_m & \leq t_1+\sum_{i=1}^{m-1}\frac{2^{n_{i}+1}\cdot n_{i}}{M_{n_{i}}}\\
           & < \frac{2^n\cdot n}{\operatorname{vol}(X)^\frac{1}{n}}+(m-1)\cdot \frac{2^{n+1}\cdot n}{M_1}\\
           &\leq \frac{2^n\cdot n}{K^\frac{1}{n}}+\cdot \frac{2^n\cdot n^2}{M_1}\\
           &\leq \frac{1}{3},
        \end{aligned}
    \end{equation*}
    where the last inequality comes from the choice of $K$ and $M_i$ at the beginning of the proof. 
    By Lemma \ref{birational}, we know the pluricanonical map $\varphi_{a,{X}}$ is birational for all $a\geq 2$.
\end{proof}

\section{Proof of Theorem \ref{main theorem2}}
We prove Theorem \ref{main theorem2} by showing that a more detailed theorem holds, which is a generalization of \cite[Theorem 4.12]{lacini2023boundedness}.
\begin{theorem}
   Let $n\geq1$ and $r\geq 2$ be positive integers.
    \begin{enumerate}
        \item There exists a $K_0>0$ such that, if $X$ is a smooth projective $n$-fold of general type with $h^0(X,rK_X)>0$, then the general fiber of $\varphi_{r,X}$ has volume less than $K_0$ (the degree of $\varphi_{r,X}$ less than  $K_0$ if $\varphi_{r,X}$ is generically finite). In particular, $r$-canonical maps have birationally bounded fibers.
        \item There exists a $K_1>0$ such that, if $X$ is a smooth projective $n$-fold of general type with $h^0(X,rK_X)=0$, then $X$ is birationally fibered in $k$-folds $Y_b$ with volume less than $K_1$ and $h^0(Y_b,rY_b)=0$ for some $1\leq k\leq n$.
    \end{enumerate} 
\end{theorem}
\begin{proof}
    (1) For a birationally bounded set $\mathcal{P}$ of smooth projective varieties of dimension $d<n$, we get a positive number $t(\mathcal{P})$ by Theorem \ref{extension}. Choose a function $\lambda:\mathbb{Z}_{>0} \times \mathbb{Z}_{>0} \rightarrow \mathbb{R}_{>0}$ satisfying $\lambda(a,b)=t(\mathfrak{X}_{a,b})$, then we get $M_{n-1}>\cdots>M_1>0$ and $K>0$ by Theorem \ref{main theorem}.

    Let $$K_0=(rn)^n\max_{1\leq k\leq n-1}\{K,M_k^k\}+1.$$
    Suppose that there exists some $X$ such that $\operatorname{vol}(F)>K_0$ where $F$ is the general fiber of $\varphi_{r,X}$ (or $\operatorname{deg}\varphi_{r,X}>K_0$ if $\varphi_{r,X}$ is generically finite). Then apply Lemma \ref{fiber} to some higher birational model of $X$ if $h^0(X,rK_X)\geq 2$, or we have $F=X$ if $h^0(X,rK_X)=1$. In both cases, $\operatorname{vol}(X)>K$. Hence either $\lvert aK_X\rvert$ induces a birational map for $a\geq 2$ or there is a commutative diagram:
    \begin{center}
        \begin{tikzcd}
            Y \arrow{r}{\pi} \arrow{d}{f} &X\\
            B
        \end{tikzcd}
    \end{center}
    where $\pi$ is birational and $f$ satisfies  $(B)_{\mathfrak{X}_{k, M_k^k}, \lambda\left(k, M_k^k\right)}$ for some $1 \leq k \leq n-1$.

    For the first case, we see that $\lvert rK_X \rvert$ induces a birational map so $F$ is a point, which is a contradiction.

    For the second case, since $\lambda(k,M_k^k)=t(\mathfrak{X}_{k,M_k^k})$ the map $$H^0(Y,rK_Y)\rightarrow H^0(Y_{b_1},rK_{Y_{b_1}})\oplus H^0(Y_{b_2},rK_{Y_{b_2}})$$ is surjective for general $b_1\neq b_2\in B$ by Theorem \ref{extension}. For general $b\in B$, $H^0(Y_b,rK_{Y_b})\neq 0$ since $H^0(Y,rK_Y)=H^0(X,rK_X)\neq 0$. Therefore, $\lvert rK_X \rvert$ can separate $\pi(x_1)$ and $\pi(x_2)$ for general $x_1\in Y_{b_1}$ and $x_2\in Y_{b_2}$ with $b_1\neq b_2$.

    For general $b\in B$, we have $h^0(Y_b,rK_{Y_b})\neq 0$ and $$\operatorname{vol}(Y_b)\leq M_k^k<\frac{K_0}{(rn)^n}.$$ By Lemma \ref{fiber}, we know $\operatorname{vol}(F^\prime)<K_0$, where $F^\prime$ is the general fiber of $\varphi_{r,Y_b}$. Since $\operatorname{vol}(F)>K_0$, we know the general fiber of $\varphi_{r,X}$ is not contained in the image via $\pi$ of any fiber of $f$. Hence, for general points $x,y\in Y$ such that $\varphi_{r,X}(\pi(x))=\varphi_{r,X}(\pi(y))$, they are in different fibers of $f$. Since $\lvert rK_X \rvert$ can separate $\pi(x_1)$ and $\pi(x_2)$ for general $x_1\in Y_{b_1}$ and $x_2\in Y_{b_2}$ with $b_1\neq b_2$, we get a contradiction.

    Therefore, for any smooth projective $n$-fold of general type with $h^0(X,rK_X)>0$, the general fiber of $\varphi_{r,X}$ has volume less than $K_0$. The last assertion follows from Theorem \ref{HaconMckernanbdd}.
    
    (2) Let $M_i,K$ be as above. Let $$K_1=\max_{1\leq k\leq n-1}\{M_k^k,K\}+1.$$ For a smooth projective $n$-fold $X$ of general type with $h^0(X,rK_X)=0$ and $\operatorname{vol}(X)>K$, $\varphi_{r,X}$ cannot be birational so there is a commutative diagram:
    \begin{center}
        \begin{tikzcd}
            Y \arrow{r}{\pi} \arrow{d}{f} &X\\
            B
        \end{tikzcd}
    \end{center}
    where $\pi$ is birational and $f$ satisfies  $(B)_{\mathfrak{X}_{k, M_k^k}, \lambda\left(k, M_k^k\right)}$ for some $1 \leq k \leq n-1$. Similarly to (1) there is a surjection $$H^0(Y,rK_Y)\rightarrow H^0(Y_{b_1},rK_{Y_{b_1}})\oplus H^0(Y_{b_2},rK_{Y_{b_2}})$$ for general $b_1\neq b_2\in B$. Therefore, $H^0(Y_{b},rK_{Y_{b}})=0$ for general $b\in B$ and $$\operatorname{vol}(Y_b)\leq M_k^k< K_1.$$

    For those $X$ with $\operatorname{vol}(X)\leq K$, the conclusion is trivial.
\end{proof}

\addtocontents{toc}{\protect\setcounter{tocdepth}{0}}
\section*{\bf Acknowledgments}
The author appreciates his Ph.D. advisor Prof. Meng Chen for introducing him to the problem. He is also grateful to Zhi Jiang, Chen Jiang, Hexu Liu for useful discussions. This work was supported by NSFC for Innovative Research Groups \#12121001.
\addtocontents{toc}{\protect\setcounter{tocdepth}{1}}
\bibliographystyle{alpha}
\bibliography{reference.bib}

\end{document}